\documentclass[10pt, reqno,amscd,amssymb]{amsart}
\usepackage[english]{babel}
\usepackage{amsfonts,amssymb,graphicx,amscd}
\usepackage{amsmath,amsthm}
\usepackage{hyperref}
\usepackage{extpfeil}
\usepackage{mathrsfs}
\allowdisplaybreaks

\newtheorem{thm}{Theorem}[section] 

\newtheorem{assumption}[thm]{Assumption}
\newtheorem{lem}[thm]{Lemma}
\newtheorem{cor}[thm]{Corollary}

\newtheorem{prop}[thm]{Proposition}

\newtheorem*{lus}{Lusztig's conjecture}

\theoremstyle{definition}

\newtheorem{defn}[thm]{Definition}

\newtheorem{rem}[thm]{Remark}

\numberwithin{equation}{subsection}
\numberwithin{thm}{section}

\newcommand{\id}{\operatorname{Id}}
\newcommand{\Hom}{\text{\rm Hom}}

\newcommand{\Ext}{\operatorname{Ext}}

\newcommand{\N}{{\mathbb N}}

\newcommand{\Z}{\mathbb Z}
\newcommand{\inv}{^{-1}}
\newcommand{\R}{\mathbb R}

\newcommand{\fg}{{\mathfrak{g}}}

\newcommand{\tothewall}{T^{\mu}_{\lambda}}
\newcommand{\fromthewall}{T_{\mu}^{\lambda}}

\newcommand{\hd}{\operatorname{hd}}
\newcommand{\rad}{\operatorname{rad}}
\newcommand{\soc}{\operatorname{soc}}

\newcommand{\stab}{\operatorname{Stab}}

\newcommand{\ch}{\operatorname{ch}}

\let\tothe\tothewall
\let\fromthe\fromthewall

\newcommand{\wdel}{{\widetilde{\Delta}}}

\newcommand{\surj}{\rightarrow\mathrel{\mkern-14mu}\rightarrow}

\newcommand{\w}{\widetilde}
\newcommand{\zd}{\Delta^0}
\newcommand{\zn}{\nabla_0}
\newcommand{\zp}{P^0}

\newcommand{\sco}{\mathscr O}
\newcommand{\amd}{A^\lambda\textmd{-mod}}
\newcommand{\azmd}{A^\lambda_0\textmd{-mod}}
\newcommand{\bmd}{A^\mu\textmd{-mod}}
\newcommand{\bzmd}{A^\mu_0\textmd{-mod}}
\newcommand{\pr}{\operatorname{pr}}
\newcommand{\gmd}{G\textmd{-mod}}
\newcommand{\umd}{U_\zeta\textmd{-mod}}

\title{Grade zero part of forced graded algebras}
\author{Hankyung KO}
\address{Mathematisches Institut, Universit{\"a}t Bonn, 
53115 Bonn, Germany} \email{hankyung@math.uni-bonn.de}

\begin{document}
\begin{abstract}
The paper concerns a certain subcategory of the category of representations for a semisimple algebraic group $G$ in characteristic $p$, which arise from the semisimple modules for the corresponding quantum group at a $p$-th root of unity. 
The subcategory, thus, records the cohomological difference between quantum groups and algebraic groups.
We define translation functors in this category and use them to obtain information on the irreducible characters for $G$ when the Lusztig character formula does not hold.


\end{abstract}
\maketitle
\section{Introduction}

Parshall-Scott\cite{psforcedintegqha}, \cite{PS14} introduced a ``forced grading'' in trying to model the representation theory of algebraic groups in prime characteristics using the better known representation theory of quantum groups at roots of unity. 
Our particular attention is on the ``grade zero part'' of the forced grading (called $A_0$-mod in the paper). 
This can be thought of as the modules over an algebraic group $G$ that correspond to semisimple modules for the root of unity quantum group associated to the root system of $G$, and measures the cohomological difference between the two representation theories. 
The cohomology of the latter is well understood via the Kazhdan-Lusztig theory, while it is not so for $G$. 
Thus, the grade zero part can be thought of as cohomological complement of the Kazhden-Lusztig theory in the category of (rational) representations for $G$. 
In particular, the composition multiplicities in the grade zero part give the expression of an irreducible character for the quantum group (known by Lusztig's conjecture) in terms of the irreducible characters for the algebraic group (which are much less known and of great interest for many decades).

In the regular blocks, which consist of the modules with composition factors having regular highest weights, a work of Parshall-Scott \cite{psforcedintegqha}\footnote{The main result \cite[Theorem 6.3]{psforcedintegqha} of \cite{psforcedintegqha} is in fact conjectural, as errors were found recently in its proof and that of a supporting Lemma \cite[Lemma 4.10]{psforcedintegqha}. We thank the authors of \cite{psforcedintegqha} for letting us know.}\label{pserror} tells us that the grade zero part has a highest weight structure.
This paper begins the study of the grade zero part of singular blocks. 
We use the translation functors in order to take advantage of what is known in the regular case. 
Section \ref{strans} defines translation ``to the wall'' and translation ``out of the wall'' functors in the grade zero category. 
In fact, the usual translation to the wall functors turn out to restrict to the grade zero part. 
Then the left and right adjoints of translation to the wall functors play the role of translations out of the wall in our setting. 
Consideration of the grade zero translation functors also leads us to a useful application of the theory to the irreducible characters for algebraic groups involving the right descent sets. See Proposition \ref{mult}. The proposition is upgraded, under an extra assumption (Assumption \ref{assum}) whose investigation we leave to a future work, to what we would call a new linkage principle (Proposition \ref{newlinkage}). 
Also, under the same assumption, the grade zero part of the regular blocks decomposes into very small subcategories (Proposition \ref{regdecomp}), and translation to the wall functors are embeddings of some of these subcategories into the grade zero part of the singular blocks (Proposition \ref{into the singular category}).

The paper is benefited from discussions with Brian Parshall, Leonard Scott, and Geordie Williamson.

\section{Preliminaries}
Let $p$ be an odd prime number. 
Let $G$ be a semisimple simply connected split algebraic group over a field $k$ of characteristic $p$. 
We also assume that $k$ is algebraically closed. 
Fix a maximal torus $T$ and a Borel subgroup $B$ in $G$. 
Let $R$ be the root system, and $\Pi$ be the set of simple roots. 
Let $X=X(T)$ be the set of weights, $X^+$ be the set of dominant weights, $W_p$ be the affine Weyl group acting on $X$ by the dot action $w.\gamma=w(\gamma+\rho)-\rho$ where $\rho$ is the sum of all fundamental weights. (See Appendix for more details.)

Associated to $R$, one defines the quantized enveloping algebra, or the quantum group, $U_v$ over the field $\mathbb Q(v)$. We refer to \cite[II.H.2]{J} for its presentation by generators and relations.
To consider the root of unity specialization, we consider the integral form $U_\mathcal{A}$, ``the Lusztig form'' in $U_v$ over the ring $\mathcal A=\Z[v,v\inv]\subset \mathbb Q(v)$.
Now let $\zeta$ be a primitive $p$-th root of unity in some field $K$. Via the map $v\mapsto \zeta$, we obtain $U_\zeta=U_\mathcal{A}\otimes K$, the quantum group at a $p$-th root of unity. 
See \cite[II.H]{J} for more details.

 In order to connect the representation theories of $G$ and $U_\zeta$ directly, we choose our field $K$ and $k$ so that there is an integral ring in between, which provides a desired connection. That is, we start with a $p$-modular system $(K,\sco, k)$ to define $G$ and $U_\zeta$. This means that $\sco$ is a discrete valuation ring with quotient field $K$ (of characteristic zero) and residue field $k$ of characteristic $p$. For example, the ring $\sco=\mathscr A_{(v-1,p)}/(1+v+\cdots +v^{p-1})$ forms a $p$-modular system.

Then $G$ and $U_\zeta$ are related via the integral form $\w U_\zeta$ over $\sco$ of $U_\zeta$: we have $\w U_\zeta\otimes_\sco k\cong$Dist($G$). We identify the weight lattices for $U_\zeta$ and $\w U_\zeta$ (with respect to $U^0_\zeta$ and $\w U^0_\zeta$) with $X=X(T)$. We do the same for the dominant weights $X^+$ and the affine Weyl group $W_p$.

Consider the category $\operatorname{Rep}(G)$ of rational representations of $G$ and $U_\zeta$-mod (resp., $\w U_\zeta$-mod) of integrable finite dimensional $U_\zeta$-modules (resp., $\w U_\zeta$-modules) of type 1. 
The categories $\operatorname{Rep}(G)$ and $U_\zeta$-mod are highest weight categories, in the sense of Cline-Parshall-Scott \cite{CPShwc}, with the infinite poset $(X^+,\uparrow)$. We provide a definition (Definition \ref{uparrow}) and some discussion on the uparrow ordering $\uparrow$ in Appendix. We denote by $\Delta(\gamma)$ the standard object in $\operatorname{Rep}(G)$ of highest weight $\gamma\in X^+$, by $\nabla(\gamma)$ the costandard object, by $L(\gamma)$ the simple object. We use the notation  $\Delta_\zeta(\gamma),\nabla(\gamma),L_\zeta(\gamma)$ for the standard, costandard, simple object of highest weight $\gamma\in X^+$ in $U_\zeta$-mod.

\subsection{Reduction mod p}
 Standard and costandard modules in the two highest weight category $\gmd$ and $\umd$ are related via the integral version as follow. Choose a minimal admissible (i.e., $\sco$-free and $\w U_\zeta$-invariant) lattice $\w\Delta(\gamma)$ in $\Delta(\gamma)$.
This is done simply by picking a highest weight vector $v$ in $\Delta_\zeta(\gamma)$ and letting $\w\Delta(\gamma):=\w U_\zeta v$. For the costandard modules, we dualize this to take an admissible lattice $\w\nabla(\gamma)$ in $\nabla(\gamma)$ rather than discussing what a maximal lattice is. 
Then we have
 \begin{equation}\label{deltalattice}
 \w\Delta(\gamma)_K\cong\Delta_\zeta(\gamma),\ \ \ \ \ \ \ \w\nabla(\gamma)_K\cong\nabla_\zeta(\gamma). \end{equation}
 and
\begin{equation}\label{deltareduction}
\w\Delta(\gamma)_k\cong\Delta(\gamma),\ \ \ \ \ \ \ \ \w\nabla(\gamma)_k\cong\nabla(\gamma).
\end{equation}
(We write $M_K:=M\otimes_\sco K$, $M_k:=M\otimes_\sco k$ for an $\sco$-module $M$.) This procedure is called ``reduction mod $p$''.

The simple objects of the two highest weight categories are quite different. Recall the Steinberg tensor product theorems for the two cases:
\begin{equation}\label{steinG}
L(\gamma_0+p\gamma_1)\cong L(\gamma_0)\otimes L(p\gamma_1)\cong L(\gamma_0)\otimes L(\gamma_1)^{[1]},
\end{equation}
\begin{equation}\label{st quan}
L_\zeta(\gamma_0+p\gamma_1)\cong L_\zeta(\gamma_0)\otimes L_\zeta(p\gamma_1)\cong L_\zeta(\gamma_0)\otimes V(\gamma_1)^{[1]},
\end{equation}
where $\gamma_0\in X_p:=\{\gamma\in X^+\ |\ \langle \gamma
+\rho,\alpha^\vee\rangle <p,\ \ \forall \alpha\in \Pi\}$, $\gamma_1\in X^+$, the module $V(\gamma_1)$ is the simple module for $U(\fg)$ of highest weight $\gamma_1$, and $-^{[1]}$ are the Frobenius twists. 
Since the simple module $L(\gamma_1)$ has a character smaller than that of $V(\gamma_1)$, which has the Weyl character, we see that the irreducibles for $U_\zeta$ are in general larger than the irreducibles for $G$.

We now construct a class of highest weight $G$-modules whose characters agree with those of irreducible $U_\zeta$ modules. We do it by reducing the modules $L_\zeta(\gamma)$ mod $p$. 
Take a minimal admissible lattice $\w L^{\text{min}}(\gamma)$ in $L_\zeta(\gamma)$ and its dual $\w L^{\text{max}}(\gamma)$. Then define the $G$-modules\footnote{These modules are denoted by $\Delta^{\text{red}}(\gamma)$ and $\nabla_{\textmd{red}}(\gamma)$ in \cite{CPS7} and many other papers. It becomes clear in the next section why we use ``$0$'' in the notation.} 
\[\zd(\gamma):=\w L^{\text{min}}(\gamma)\otimes k\] and \[\zn(\gamma):=\w L^\text{max}(\gamma)\otimes k.\] These modules are not irreducible in general. We have by construction and \eqref{deltareduction}
$$\Delta(\gamma)\surj\zd(\gamma)\surj L(\gamma)$$ and
$$\nabla(\gamma)\surj\zn(\gamma)\surj L(\gamma).$$
For more properties of these reduction mod $p$ modules that we do not need in the paper, we refer the reader to \cite{CPS7}.

\subsection{Linkage principle}\label{link}
We review the linkage principle for $G$ and for $U_\zeta$. Let $C^-$ be the top antidominant alcove, and write $C^-_\Z:=C^-\cap X$. 
The affine Weyl group $W_p$ is generated by the reflections through the walls of $C^-$. These simple reflections form a system of Coxeter generators of $W_p$, which we denote by $S_p$.
Given a weight $\gamma$, we can find a unique weight $\mu\in\overline{C^-_\Z}$ and a (unique if and only if $\mu\in C^-$) Weyl group element $w\in W_p$ such that $\gamma=w.\mu$. Let $$W^+=\{w\in W_p\ |\ w.{C^-_\Z}\subset X^+\}.$$ 
By definition, the dominant weights in $W_p.\mu$ are exactly the elements in $W^+.\mu$. 
If $\mu$ is regular, i.e., $\mu\in C^-$, then the dominant weights in $W_p.\mu$ are in one to one correspondence with $W^+$. 
If $\mu$ is not regular (we call it singular), then we need more notations. Let $$J=S_p(\mu):=\{s\in S_p\ |\ s.\mu=\mu\},$$ and $W_J$ be the subgroup of $W_p$ generated by $J$.
So $W_J=\stab_{W_p}(\mu)$. Let $W^J$ be the set of minimal length representatives in $W_p/W_J$. Now we identify the dominant weights in $W_p.\mu$ with $$W^+(\mu):=W^+\cap W^J.$$

By the linkage principle \cite[II.6]{J}, \cite[\S 8]{APW}, the $G$-modules and $U_\zeta$-modules decomposes into the submodules (which are summands) having highest weights in the same $W_p$ orbits.
Using our notation, we can write this as the decomposition
$$\operatorname{Rep}(G)=\bigoplus_{\mu\in\overline{C^-_\Z}}(\operatorname{Rep}(G))[W^+(\mu).\mu]$$
and
$$\umd=\bigoplus_{\mu\in\overline{C^-_\Z}}(\umd)[W^+(\mu).\mu].$$
(Given a highest weight category $\mathcal C$ with a poset $\Lambda$ and an ideal $\Gamma \unlhd\Lambda$, we set $\mathcal C[\Gamma]$ to be the Serre subcategory of $\mathcal C$ generated by the irreducibles in $\{L(\gamma)\}_{\gamma\in\Gamma}$.)
It is now a standard fact that $\w U_\zeta$-mod also decomposes into $W_p$-orbits.

\subsection{Finite dimensional algebras}
We prefer working with highest weight categories that have finite posets. So take a finite ideal $\Gamma \subset X^+$. 
Then there is a finite dimensional $k$-algebra $A$ such that $A$-mod is equivalent to ($\operatorname{Rep}(G)$)$[\Gamma]$ and a finite dimensional $K$-algebra $A_\zeta$ such that $A_\zeta$-mod is equivalent to ($\umd)[\Gamma]$. We denote the standard, costandard, irreducible $A$-modules and $A_\zeta$-modules by the same notation as the corresponding $G$-modules and $U_\zeta$-modules. 

We can choose the algebras $A$, $\w A$, $A_\zeta$ so that 
\begin{equation}\label{algebrachoice}
A \cong \w A\otimes_\sco k\ \ \ \ \ \ \ \textmd{and}\ \ \ \ \ \ \ A_\zeta\cong \w A\otimes_\sco K
\end{equation} and the reduction mod $p$ from $U_\zeta$-modules to $G$-modules agrees with reduction mod $p$ from $A_\zeta$-modules to $A$-modules.

Considering finite dimensional algebras instead of the entire representation categories further provides finite dimensional projective modules.
We denote by $P(\gamma)$ the projective cover of the simple $A$-module $L(\gamma)$ of highest weight $\gamma\in \Gamma$. Similarly, $P_\zeta(\gamma)$ denotes the projective cover of $L_\zeta(\gamma)$ in $A_\zeta$-mod. These are all indecomposable projectives in both categories. Similarly denote the indecomposable injective modules by $I(\gamma)$ and $I_\zeta(\gamma)$.

By \S\ref{link} and \eqref{algebrachoice}, we have $$A=\bigoplus_{\mu\in\overline{C^-_\Z}}A^\mu,\ \ \ \ \ \ \ \ \ \ \w A=\bigoplus_{\mu\in\overline{C^-_\Z}}\w A^\mu,\ \ \ \ \ \ \ \ \ \ A_\zeta=\bigoplus_{\mu\in\overline{C^-_\Z}} A_\zeta^\mu,$$
where 
\[A^\mu\operatorname{-mod}\cong (\operatorname{Rep}(G))[\Gamma\cap W_p.\mu],\]
\[\w A^\mu\operatorname{-mod} \cong(\w U_\zeta\operatorname{-mod})[\Gamma\cap W_p.\mu],\]
\[A_\zeta^\mu\operatorname{-mod}\cong(\umd)[\Gamma\cap W_p.\mu],\]
and the algebras are related by
\[A^\mu\cong \w A^\mu\otimes_\sco k,\ \ \ \ \ \ \ \ \ \ \ \ \ A^\mu_\zeta\cong \w A^\mu\otimes_\sco K.\]

We may assume that the poset ideal $\Gamma$ satisfies $\Gamma\cap W_p.\lambda=\{w.\lambda\ |\ w\in W^+, w\leq vw_0\}$ for some $v\in W^+(\mu):=W^+\cap W^J$, where $w_J$ is the maximal element in $W_J$ and $w_0$ the maximal element in the finite Weyl group $W_f$. Note that $J$ is always a proper subset of $S_p$ and thus $W_J$ is a finite group.
In fact, $\Gamma$ is determined by this condition, depending on the choice of $v$.
This assumption ensures that the translation functors (see \S\ref{sstr}) are well defined between $A^\mu$-mod and $A^\lambda$-mod for any two weights $\lambda,\mu\in \overline{C^-_\Z}$ and satisfy basic properties as in Proposition \ref{tr}.


\subsection{Translation functors}\label{sstr}
The translation functors for $G$-modules and $U_\zeta$-modules readily give the translation functors for $A$-modules and $A_\zeta$-modules. We recall the definition and define the integral version of the translation functors.

Given a weight $\mu\in \overline{C^-_\Z}$, we denote by $\pr^p_\mu$ the projection 
$$\pr^p_\mu:A\textmd{-mod}\to\bmd$$ to the $\mu$ orbit. 
It maps an $A$-module $M$ to the largest direct summand of $M$ with composition factors of the form $L(w.\mu)$. 
We also have the projections $$\pr^\zeta_\mu: A_\zeta\textmd{-mod}\to A^\mu_\zeta\textmd{-mod}$$  and
$$\w\pr_\mu:\w A\textmd{-mod}\to \w A^\mu\textmd{-mod}$$ to the $\mu$ orbit. 
Since the linkage classes for the three cases are the same, we have 
\begin{equation}\label{tensorcommute}
\w\pr_\mu(-)\otimes_\sco k\cong \pr_\mu^p(-\otimes_\sco k)\ \ \textrm{ and }\ \ \w\pr_\mu(-)\otimes_\sco K\cong \pr^\zeta_\mu(-\otimes_\sco K).
\end{equation}

Now fix $\lambda,\mu\in \overline{C^-_\Z}$. The translation $$\tothe:\amd \to \bmd$$
and
$$\tothe:A^\lambda_\zeta\textrm{-mod} \to A^\mu_\zeta\textmd{-mod}$$
are defined as 
$$\tothe=\pr^p_\mu(-\otimes_k\Delta(\nu))\ \ \ \textrm{and}\ \ \ \ \tothe=\pr^\zeta_\mu(-\otimes_K\Delta_\zeta(\nu)),$$ 
where $\nu$ is the unique element in $W(\mu-\lambda)\cap X^+$. 
It will be clear from the context whether $\tothe$ is a functor on $A$-modules or on $A_\zeta$-modules.
The translation functors form adjoint pairs $(\tothe,\fromthe)$ and $(\fromthe,\tothe)$, are exact, preserve projectives and injectives.

The functors $\tothe$ and $\fromthe$ are easiest to study in case $\mu$ is in the closure of the facet containing $\lambda$. We list some important properties in this case. See \cite[II.7]{J} for more details. For simplicity, we assume $\lambda\in C^-$.

\begin{prop}\label{tr} Let $y\in W^+(\mu)$ such that $y.\lambda \in \Gamma$. So $y.\mu$ is in the upper closure of the facet containing $y.\lambda$. Let $x\in W_J$ and assume that $yx.\lambda, y.\mu \in \Gamma$.
\begin{enumerate}
\item $\tothewall\Delta(yx.\lambda)=\Delta(yx.\mu)=\Delta(y.\mu)$.
\smallskip
\item $\fromthewall\Delta(y.\mu)$ has a $\Delta$-filtration whose sections are exactly $\Delta(yz.\lambda)$ where each $z\in W_J$ occurs with multiplicity one.
\smallskip
\item $\tothewall L(yx.\lambda)=
\begin{cases}
L(y.\mu)\ \ \ \ \ \ \ \ \ \text{if } x=e,\\
0\ \ \ \ \ \ \ \ \ \ \ \ \ \ \ \ \ \text{otherwise}.
\end{cases}$
\smallskip
\item $\hd(\fromthewall L(y.\mu))\cong L(y.\lambda), \ \ \soc(\fromthewall L(y.\mu))\cong L(y.\lambda).$
\item $\tothewall P(y.\lambda)= P(y.\mu)^{\oplus |W_J|}$.
\smallskip
\item $\fromthewall P(y.\mu)=P(y.\lambda)$.
\end{enumerate}
We also have the dual statements and the quantum analogues, which we omit.
\end{prop}

\begin{proof}
We prove (6) and (5). See \cite[II.7.11, 7.13, 7.15, 7.20]{J} for (1)-(4).

For (6), we already know that $\fromthe P(y.\mu)$ is projective. Also, $$\Hom_{A^\lambda}(\fromthe P(y.\mu), L(w.\lambda))\cong \Hom_{A^\mu}(P(y.\mu),\tothe L(w.\lambda))=
\begin{cases}
k\ \ \ \text{if } y=w,\\
0\ \ \ \text{otherwise}
\end{cases} $$ by (3). Thus the only possibility is that $\fromthe P(y.\mu)\cong P(y.\lambda)$

Now we prove (5). Again, we know that $\tothe P(y.\lambda)$ is a projective, so a direct sum of indecomposable projective. It remains to determine its character. By (1) and (2), $\ch \tothe\fromthe M= |W_J|\ch M$ for any $A^\mu$-module $M$. 
Applying this to $M=P(y.\mu)$ and using (6), we have $$\ch\tothe P(y.\lambda)=\ch\tothe\fromthe P(y.\mu)= |W_J|\ch P(y.\mu).$$ We necessarily have $\tothewall P(y.\lambda)\cong P(y.\mu)^{\oplus |W_J|}$.
\end{proof}

We can define the integral translation functor $\w\tothe$ in the same way, that is, 
\[\w\tothe:= \w\pr_\mu(-\otimes_\mathscr{O} \wdel(\nu))\] 
where $\nu$ is the unique element in $X^+\cap W(\mu-\lambda)$ and $\wdel(\nu)$ is a minimal admissible lattice of $\Delta_\zeta(\nu)$. The functor $\w\tothe$ is, among other things, exact, since $\w\Delta(\nu)$ is $\mathscr O$-free. Note also that $\w\tothe$ maps $\sco$-free modules to $\sco$-free modules.

The following two lemmas do not assume that $\lambda$ is regular. In particular, $\lambda$ and $\mu$ can be interchanged. We say $\w M$ is a lift of $M\in A^\lambda$-mod if $\w M$ is (identified with) some not-necessarily-$\mathscr O$-free $\w U_\zeta$-module such that $\w M_k:=\w M\otimes_{\mathscr O} k\cong M$.

\begin{lem}\label{trcommute}
Let $\w M$ be a lift of an $A^\lambda$-module $M$. Then $\w\tothe \w M$ is a lift of $\tothe M$.
\end{lem}
\begin{proof}
We have
\begin{align*}
(\w \tothe \w M)_k&= (\w\pr_\mu(\w M\otimes_\mathscr{O} \wdel(\nu)))\otimes_\sco k\\
&\cong \pr^p_\mu((\w M\otimes_\sco \wdel(\nu))\otimes_\sco k)\\
&\cong \pr^p_\mu((\w M\otimes_\sco k)\otimes (\wdel(\nu)\otimes_\sco k))\\
&\cong \pr^p_\mu(M\otimes_k \Delta(\nu))\\
&=\tothe M
\end{align*}
by \eqref{tensorcommute}, \eqref{deltareduction}, and definitions of translation functors.
\end{proof}

\begin{lem}\label{trcom2}
For an $\w A^\lambda$-module $\w M$, we have $$(\w\tothe \w M)_K\cong \tothe \w M_K,$$ where the second $\tothe$ is the translation functor in $A_\zeta$-modules.
\end{lem}
\begin{proof}
By \eqref{tensorcommute} and \eqref{deltalattice}, we have
\begin{align*}
(\w \tothe \w M)_K&= (\w\pr_\mu(\w M\otimes_\mathscr{O} \wdel(\nu)))\otimes_\sco K\\
&\cong \pr^\zeta_\mu((\w M\otimes_\sco \wdel(\nu))\otimes_\sco K)\\
&\cong \pr^\zeta_\mu((\w M\otimes_\sco K)\otimes (\wdel(\nu)\otimes_\sco K))\\
&\cong \pr^\zeta_\mu(\w M_K\otimes_K \Delta_\zeta(\nu))\\
&=\tothe \w M_K.
\end{align*}
\end{proof}

\subsection{Grade zero part of the finite dimensional algebras}
Parshall and Scott introduced a ``forced graded algebra'' associated to the algebra $A$. The definition goes as $$\w{\textmd{gr}}A=\bigoplus_i(\w A\cap \rad^i A_\zeta/\w A\cap \rad^{i+1} A_\zeta)_k.$$
What we do here is to use the radical grading on $A_\zeta$, (i.e., the grading defined by $A_\zeta^i=\rad^i A_\zeta/\rad^{i+1} A_\zeta$) which behaves extremely well (it is standard Koszul; see \cite[\S 6]{PS15} and \cite{svv}), to define an associated graded algebra of $A$. 
We do not know if the algebra $A$ itself is graded (i.e., $\w{\textmd{gr}}A\cong A$).
The first conjecture of Parshall-Scott in \cite{PS14} expects $\w{\textmd{gr}}A$ to be standard Q-Koszul, under the Kazhdan-Lusztig correspondence between $\umd$ and a certain subcategory of the corresponding affine category O in a negative level.
This roughly says that the difference between the Koszul grading on $A_\zeta$ and the grading of $\w{\textmd{gr}}A$ is just in the grade zero part. 
We do not give the full definition of Q-Koszulity except that it requires the grade zero part
 $$A_0=(\w{\textmd{gr}}A)_0=(\w A/\w A\cap \rad A_\zeta)_k$$
to be quasi-hereditary. 
The algebra $A_0$ is what this paper concentrates on.
We can understand $A_0$ in terms of the module category as follows.
With the obvious map $A\surj A_0$ we can view $A_0$-mod as a full subcategory of $A$-mod consisting of the $A$-modules on which the ``forced'' radical $(\w A\cap \rad A_\zeta)_k$ of $A$ acts as zero.
A useful way to think of $A_0$-modules is to consider them as $A$-modules that are ``semisimple for $A_\zeta$''.

Some easy examples of $A_0$-modules are $L(\gamma),\zd(\gamma),\zn(\gamma)$ for $\gamma\in \Gamma$. The projective cover of $L(\gamma)$ in the category $\azmd$ is denoted by $P^0(\gamma)$, the injective envelope by $I^0(\gamma)$.
The orbit decomposition of $A$ gives
$$A_0=\bigoplus_{\mu\in\overline{C^-_\Z}}A_0^\mu.$$ 

We record the following ``regular'' result of Parshall-Scott, which is, by footnote \ref{pserror}, a conjecture now.
\begin{thm}\label{regqha}
Let $p\geq 2h-2$ and $\lambda\in C^-_\Z$. Then $A^\lambda_0$ is a quasi-hereditary algebra. In other words, $\azmd$ is a highest weight category. 
In this category $\Delta^0(w.\lambda)$ are the standard objects, $\zn(w.\lambda)$ are the costandard objects for $w\in W^+$ .
\end{thm}
\begin{proof}
This is a consequence of \cite[Theorem 6.3]{psforcedintegqha}; see \cite[
Corollary 3.2]{pspfiltr}.
\end{proof}

One of the defining properties of the highest weight category is that a projective object has filtration by standard objects. We denote by 
\[(P:\Delta^0(\lambda'))\]
the multiplicity of $\Delta^0(\lambda')$ in such a filtration of a projective $A^\lambda_0$-module $P$.
An immediate consequence of Theorem \ref{regqha} is the following reciprocity of Brauer-Humphreys type.
\begin{cor}\label{reciprocity}
Suppose $\azmd$ is a highest weight category, and $\lambda\in C^-_\Z$. Then the following holds.
\[(P(w.\lambda):\Delta^0(y.\lambda))=[\nabla_0(y.\lambda):L(w.\lambda)]=[\Delta_0(y.\lambda):L(w.\lambda)]\]
\end{cor}

\section{Translation and characters}\label{scharacter}

Here we introduce the character function to clarify some points that are already clear but hard to explain. The characters are defined on $G$-modules and $U_\zeta$-modules. By abusing notation we consider the character function for $A$-modules
$$\ch:A\textrm{-mod}\to \Z X$$
and for $A_\zeta$-modules
$$\ch:A_\zeta\textrm{-mod}\to \Z X.$$ For $\nu\in X^+$, the Weyl character \[\chi(\nu)=\ch\Delta(\nu)=\ch\Delta_\zeta(\nu)\] 
gives the characters for standard and costandard objects in the categories $A$-mod and $A_\zeta$-mod.
The images of both maps $\ch$ are in the $\Z$-span of $\{\chi(\nu)\}_{\nu\in X^+}$ in $\Z X$. Now we have the fourth projection
$$\pr_\mu:\Z\{\chi(\nu)\}_{\nu\in X^+}\to\Z\{\chi(\nu)\}_{\nu\in W^+.\mu}$$ defined in the obvious way. 

\begin{lem}\label{chpr}
We have $$\pr_\mu\circ \ch=\ch\circ\pr^p_\mu$$ and $$\pr_\mu\circ \ch=\ch\circ\pr^\zeta_\mu.$$
\end{lem}
\begin{proof}
Let $M$ be an $A$-module. We have two expressions
\begin{align*}
\ch M&=\sum_{\nu\in X^+}c_\nu\chi(\nu)=\sum_{\nu\in X^+}d_\nu\ch L(\nu),
\end{align*}
where the coefficients $c_\nu\in \Z$ and $d_\nu\in \N$ are uniquely determined by $M$.
Then 
\begin{align*}
\pr_\mu(\ch M)&=\sum_{\nu\in W^+.\mu}c_\nu\chi(\nu)
=\sum_{\nu\in W^+.\mu}d_\nu\ch L(\nu),
\end{align*}
since $\chi(\nu)$ with $\nu\in W^+.\mu$ is a linear combination of $\{\ch L(x.\mu)\}_{x\in W^+(\mu)}$.
On the other hand, $$\ch(\pr^p_\mu M)=\sum_{\nu\in W^+.\mu}d_\nu\ch L(\nu)$$ by the definition of $\pr^p_\mu$, since $d_\nu=[M:L(\nu)]$. This shows the first identity. 

We obtain a proof for the second identity from the proof for the first if we replace $L(\nu)$ by $L_\zeta(\nu)$ everywhere.
\end{proof}

\begin{lem}\label{trch}
Suppose $M$ is an $A^\lambda$-module, $N$ is an $A^\lambda_\zeta$-module, and $\ch M= \ch N$ in $\Z X$. Then for any $\lambda,\mu\in \overline{C^-_\Z}$, we have the equality $\ch\tothe M=\ch\tothe N$ in $\Z X$.
\end{lem}
\begin{proof}
Letting $\nu$ be the unique element in $W(\mu-\lambda)\cap X^+$,
\begin{align*}
\ch\tothe M&=\ch(\pr^p_\mu(M\otimes_k\Delta(\nu)))\\
&=\pr_\mu(\ch(M\otimes_k\Delta(\nu)))\\
&=\pr_\mu(\ch M \times \chi(\nu))\\
&=\pr_\mu(\ch N\times\chi(\nu))\\
&=\pr_\mu(\ch(N\otimes_K\Delta_\zeta(\nu)))\\
&=\ch(\pr^\zeta_\mu(N\otimes_K\Delta_\zeta(\nu))\\
&=\ch\tothe N,
\end{align*}
where the second and sixth equality follows by Lemma \ref{chpr}, and the rest are either obvious or by definition.
\end{proof}

Denote by $\widetilde P(\gamma)$ the projective cover of $L(\gamma)$ in $\widetilde A$-mod. This is an $\sco$-free module such that $\widetilde P(\gamma)_k\cong P(\gamma)$, and has a character equal to the character of $P(\gamma)\in A$-mod. Note that while $\widetilde P(\gamma)_k\cong P(\gamma)$, we have $\widetilde P(\gamma)_K\cong P_\zeta(\gamma)\oplus P'$ where $P'$ is a direct sum of some $P_\zeta(\gamma')$ with $\gamma<\gamma'\in\Gamma$. (See \cite{psforcedintegqha} for more information.) We make the following observation on the summands of $P'$.

\begin{prop}\label{projective}
Let $\lambda\in \overline{C^-_\Z}$. Writing
\begin{equation}\label{projdecomp}
\widetilde P(w.\lambda)_K\cong P_\zeta(w.\lambda)\oplus \bigoplus_{w<y\in W^+} P_\zeta(y.\lambda)^{\oplus n_y},
\end{equation}
we have $n_{ws}=0$ for all $s\in S$ (with $w<ws\in\Gamma$).
\end{prop}
\begin{proof}
Let $s\in S$ be such that $w<ws$.
Write 
\[\ch P_\zeta(w.\lambda)=\sum_xc_x\chi (x.\lambda),\ \ \ \ch P(w.\lambda)=\sum_xd_x\chi(x.\lambda) .\]
Then 
\[c_{ws}=(P_\zeta(w.\lambda):\Delta_\zeta(ws.\lambda))=[\Delta_\zeta(ws.\lambda):L_\zeta(w.\lambda)]=1\]
and
\[d_{ws}=(P(w.\lambda):\Delta(ws.\lambda))=[\Delta(ws.\lambda):L(w.\lambda)]=1.\]
The last equality follows, for example, from that there is a unique (up to scalar) homomorphism from $\Delta(w.\lambda)$ to $\Delta(ws.\lambda)$ (The dual of \cite[II.7.19.d)]{J}). Since $w$ is maximal among the highest weights of the composition factors in $\rad\Delta(ws.\lambda)$, the universal property of the highest weight module $\Delta(w.\lambda)$ applies. The quantum case is the same.

Since $\ch P(w.\lambda)= \ch \widetilde P(w.\lambda)=\ch \widetilde P(w.\lambda)_K$, the equation \eqref{projdecomp} gives
\[\sum_xd_x\chi(x.\lambda)= \sum_xc_x\chi (x.\lambda)+ \sum_{w<y\in W^+} \ch P_\zeta(y.\lambda)^{\oplus n_y}.\]
By $c_{ws}=d_{ws}$ (and that $\{\chi(\gamma)\}$ forms a basis for the character ring), the projectives appearing in the second summand of the right hand side cannot contain $\Delta_\zeta(ws.\lambda)$ as a subquotient. This proves the claim.
\end{proof}

\section{Translating $A_0$-modules}\label{strans}

We define translation functors between $A_0$-module categories.
In this section, $\lambda\in C^-_\Z$ is regular, and $\mu\in \overline {C^-_\Z}$ is singular.
Then the next proposition shows that the functor $\tothe$ simply restricts to $A_0$-modules. More generally, if $\mu$ lies in the closure of the facet containing $\lambda$, then by the same proof the translation functor $\tothe$ is restricted to $A_0$-modules.

\begin{prop}\label{ztothe}
The translation $\tothe:A^\lambda$-mod$\to A^\mu$-mod maps $A^\lambda_0$-modules to $A^\mu_0$-modules. That is, we have a translation functor $$\tothe|_{A^\lambda_0\text{-mod}}:A^\lambda_0\textrm{-mod}\to A^\mu_0\textrm{-mod}$$
which we just denote by $\tothe$.
\end{prop}

\begin{proof}
Since $\tothe$ is exact, it is enough to show that $\tothe P^0$ is an $A^\mu_0$-module for any projective $A^\lambda_0$-module $P^0$. We may assume that $P^0=P^0(w.\lambda)$. 

Let $\w\zp$ be a lift of $\zp$. This is just the degree zero part of the integral projective $\w P(w.\lambda)$ which is the projective cover of $L(w.\lambda)$ in $\w A^\lambda$-mod. (See \cite[p.7]{psforcedintegqha}.) In particular, $\w\zp_K$ is semisimple. Since $\tothe$ sends irreducibles to irreducibles, we know that $\tothe(\w\zp_K)$ is semisimple. But by Lemma \ref{trcom2}, $\tothe(\w\zp_K)\cong (\w\tothe\w\zp)_K$ as $\w A^\mu_K$-modules. This shows that $\w\tothe\w\zp$ is an $\w A^\mu_0$-module. Thus, by Lemma \ref{trcommute}, $\tothe \zp$ is an $A^\mu_0$-module.
\end{proof}

\begin{prop}\label{tothe}
Let $w\in W^+$ and $J=S_p(\mu)$. We have $$\tothewall \Delta^0(w.\lambda)=
\begin{cases}
\Delta^0(w.\mu)\ \ \ \ \ \ \ \ \ \ \ \ \text{if } w\in W^J, \\
0\ \ \ \ \ \ \ \ \ \ \ \ \ \ \ \ \ \ \ \ \ \ \text{otherwise},
\end{cases}.$$ The same is true for $\zn(w.\lambda)$.
\end{prop}
\begin{proof}
The second case follows from Lemma \ref{trch} because $\tothe L_\zeta(w.\mu)=0$ if $w\not\in W^J$. So assume $w\in W^J$. By construction, $\zd(w.\lambda)=(\w\zd(w.\lambda))_k=\w\zd(w.\lambda)\otimes_\sco k$ where $\w\zd(w.\lambda)$ is a minimal admissible lattice in $L_\zeta(w.\lambda)$. 
By Lemma \ref{trcom2}, $(\w\tothe \w\zd(w.\lambda))_K$ is isomorphic to $\tothe \w\zd(w.\lambda)_K$, which is isomorphic to
$\tothe L_\zeta(w.\lambda)\cong L_\zeta(w.\mu)$. Thus
$\w\tothe \w\zd(w.\lambda)$ is an admissible lattice in $L_\zeta(w.\mu)$. On the other hand, $\w\tothe \w\zd(w.\lambda)$ is a quotient of $\w\tothe \w\Delta(w.\lambda)$, since $\w\zd(w.\lambda)$ is a quotient of $\w \Delta(w.\lambda)$. Thus, $\w\tothe \w\zd(w.\lambda)$ is generated by a single vector, hence is a minimal lattice. Now the claim follows from Lemma \ref{trcommute}.
\end{proof}

Given $w\in W_p$, let $$R(w):=\{s\in S_p\ |\ ws<w\}.$$ This is called the right descent set of $w$. The following observation is an interesting corollary of Proposition \ref{tothe}.

\begin{prop}\label{mult}
Let $\lambda\in C^-_\Z$ and $p\geq h$. Then, for $w,x\in W^+$, $[\zd(w.\lambda):L(x.\lambda)]\neq 0$ implies $R(w) \subset R(x)$.
\end{prop}
\begin{proof}
By assumption, for each $s\in R(w)$ there exists $\mu\in \overline{C^-_\Z}$ such that $J=S_p(\mu)=\{e,s\}$ (See \cite[II.6.3]{J}). Now we consider the translation $\tothe$. Suppose $[\zd(w.\lambda):L(x.\lambda)]\neq 0$. Since $ws<w$, by Proposition \ref{tothe}, we have $\tothe\zd(w.\lambda)=0$. By exactness of $\tothe$, the module $\tothe L(x.\lambda)$ should also be zero. This implies $xs<x$, hence $s\in R(x)$. The claim follows.
\end{proof}
\begin{rem}
The proposition and the proof generalize to the case where $\lambda$ is not regular provided that we can find enough singular weights on the wall of the facet of $\lambda$. 
\end{rem}

Let us now construct the translation functors of the opposite direction in $A_0$-module categories. By general category theory (adjoint functor theorem), we already know that the functor $\tothe:\azmd\to\bzmd$ has a left adjoint and a right adjoint. But we also know what they are. 
Consider the inclusion $\iota:\azmd\to\amd$ induced by the obvious map $A^\lambda\surj A^\lambda_0$. The left adjoint \[L:\amd\to \azmd\] takes an $A^\lambda$-module $M$ to its largest quotient which is an $A^\lambda_0$-module, and the right adjoint \[R:\amd\to\azmd\] takes an $A^\lambda$-module $M$ to its largest submodule which is an $A^\lambda_0$-module. The same is true for $A^\mu$ and $A^\mu_0$. 
Then for $M\in A^\mu_0$-mod and $N\in\azmd$, 
\begin{align*}
\Hom_{A_0^\lambda}(L \fromthe \iota M, N)&\cong\Hom_{A^\lambda}(\fromthe \iota M,\iota N)\\
&\cong\Hom_{A^\mu}(\iota M,\tothe \iota N)\\
&\cong\Hom_{A_0^\mu}(M,R \tothe \iota N)\\
&=\Hom_{A_0^\mu}(M,\tothe N),
\end{align*}
where the last equality makes use of Proposition \ref{ztothe}.
It follows that $L\circ\fromthe\circ\iota$ is a left adjoint of $\tothe:\azmd\to\bzmd$. Similarly, $R\circ\fromthe\circ\iota$ is a right adjoint of it. 
We simply denote them by $L\fromthe$ and $R\fromthe$ (The notation cannot be confused with the left and right derived functors since $\fromthe$ is exact.):

\[
\azmd\ \ \ \begin{array}{c}
 \xleftarrow{\ \ L\fromthe\ } \\ 
 \xrightarrow{\ \ \tothe\ \ }\\
\xleftarrow{\ \ R\fromthe\ }
 \end{array}
\ \ \ \bzmd\]

We assume for the rest of the paper $A^\lambda_0$ is quasi-hereditary.

\begin{prop}\label{fromthe}
For $w\in W^+(\mu)$, we have $$L\fromthe \zd(w.\mu)\cong \zd(w.\lambda)$$ and 
$$R\fromthe \zn(w.\mu)\cong \zn(w.\lambda).$$
\end{prop}
\begin{proof}
We prove the first isomorphism. The second is proved similarly. Note that $\fromthe\zd(w.\mu)$ is indecomposable. In fact, $\hd\fromthe\zd(w.\mu)$ is irreducible since
$$\Hom_{A^\lambda}(\fromthe\zd(w.\mu),L(x.\lambda))\cong\Hom_{A^\mu}(\zd(w.\mu),\tothe L(x.\lambda))=\begin{cases} k\ \ \ \ \ \ \textrm{if $w=x$}\\
0\ \ \ \ \ \ \textrm{otherwise}
\end{cases}
$$ by Proposition \ref{tr} (3).
Also, the proof of Lemma \ref{tothe} shows that the $A^\lambda$-module $\fromthe \zd(w.\mu)$ has $\zd(w.\lambda)$ as a quotient which is almost by definition an $A_0$-module. So it remains to show that if there is an $A_0$-module $M$ such that $$\fromthe \zd(w.\mu)\surj M\surj \zd(w.\lambda),$$ then $M=\zd(w.\lambda)$.

We suppose the contrary and deduce a contradiction. That is, suppose we have a non-split short exact sequence \begin{equation}\label{K}
0\to K\to M\to \zd(w.\lambda)\to 0
\end{equation} in $A_0$-mod for a nonzero $A_0$-module $K$.
In other words, $M$ represent a nontrivial element in $\Ext^1_{A_0}(\zd(w.\lambda),K)$. By replacing $M$ by its $A_0$-quotient, we assume that $K$ is irreducible. 
Thus there is some $x\in W$ with $L(x.\lambda)=K$.
Since $\zd(w.\lambda)$ has the universal property in $A_0$-mod, it should be the case that $x>w$. In particular, $M$ is not a quotient of $\Delta(w.\lambda)$. 
We claim that $x=ws$ for $s\in J$. 

Let us first show that $x=wz$ for some $z\in W_J$. We know that $\fromthe \Delta(w.\mu)$ has a $\Delta$-filtration whose sections are exactly $\Delta(wy.\lambda)$ for $y\in W_J$. 
Since $\fromthe \zd(w.\mu)$ is a quotient of $\fromthe \Delta(w.\mu)$, the moudule $M$ is also a quotient of $\fromthe\Delta(w.\mu)$. This implies that there is a $\Delta$-section $\Delta(wz.\lambda)$ of $\fromthe\Delta(w.\mu)$ such that $K$ is a composition factor of $\Delta(wz.\lambda)$. But since $\zd(w.\lambda)$ is a quotient of $\Delta(w.\lambda)$, \eqref{K} shows that $z\neq e$ and $K$ is a quotient of $\Delta(wz.\lambda)$. Thus $K\cong L(wz.\lambda)$.

Now we show that $z\in J$. We have by (the proof of) \cite[Lemma 2.4]{mine} that any section of the form $\Delta(wz'.\lambda)$ with $l(z')>1$ in $\fromthe\Delta(w.\lambda)$ is generated by another section $\Delta(wz''.\lambda)$ where $z''<z'$ in $W_J$. 
This implies that $l(z)=1$, proving our claim.

Recall that $M$ is a quotient of $\fromthe \Delta(w.\lambda)$, hence is indecomposable. So there is a surjective map $P^0(w.\lambda)\surj M.$ 
Since $P^0(w.\lambda)$ has a $\zd$-filtration, this factors through
\[\zp(w.\lambda)\surj N\surj M,\]
where $N$ is an extension of $\zd(w.\lambda)$ by some section $\zd(y.\lambda)$.
The head of this $\zd(y.\lambda)$ need to contain a copy of the irreducible $K$. Therefore we have $y=ws$.
In particular,
\[[\zd(ws.\lambda):L(w.\lambda)]=(\zp(w.\lambda):\zd(ws.\lambda))\neq 0.\] 
But this contradicts Proposition \ref{mult}, since $ws>w$ and $wss<ws$.
\end{proof}

\begin{rem}
One can alternatively use Proposition \ref{projective} in the last part of the proof.
\end{rem}

\begin{prop}\label{fromthep}
For $w\in W^+(\mu)$, we have $$L\fromthe \zp(w.\mu)\cong \zp(w.\lambda).$$
\end{prop}
\begin{proof}
The functor $L\fromthe$ is a left adjoint of an exact functor, hence sends projectives to projectives. So $L\fromthe \zp(w.\mu)$ is a direct sum of indecomposable projectives $\zp(x.\lambda)$. 

But $L\fromthe\zp(w.\mu)$ is, by construction, a quotient of $\fromthe\zp(w.\mu)$ which is, by exactness of $\fromthe$, a quotient of $\fromthe P(w.\mu)\cong P(w.\lambda)$. Thus $L\fromthe\zp(w.\mu)$ is indecomposable and has the head isomorphic to $L(w.\lambda)$. This shows $L\fromthe \zp(w.\mu)\cong \zp(w.\lambda).$
\end{proof}

\section{An extra assumption}

We explore in this section some strong consequences of an extra assumption together with Theorem \ref{regqha}. 
We show later that this extra assumption is satisfied if $p$ is large enough, but we do not know whether it is true or not for general $p$.
We keep our convention $\lambda\in C^-_\Z$, $\mu\in\overline{C^-_\Z}$, and $J=s_p(\mu)$ and assume the following.
\begin{assumption}\label{assum}
If $w\in W^+\setminus W^J$, then $\tothe P^0(w.\lambda)=0$.
\end{assumption}

Then we have the other direction of Proposition \ref{mult}
to obtain a ``new linkage principle''.

\begin{prop}\label{newlinkage}
For $w,x\in W^+$, $[\zd(w.\lambda):L(x.\lambda)]\neq 0$ implies $R(w) = R(x)$.
\end{prop}
\begin{proof}
Suppose $[\zd(w.\lambda):L(x.\lambda)]\neq 0$.
The easy direction $R(w)\subset R(x)$ is established in Proposition \ref{mult}.
To show the other direction, take an arbitrary $s\in R(x)$ and choose some $\mu\in\overline{C^-_\Z}$ such that $J=S_p(\mu)={e,s}$.
Recall that $\azmd$ is highest weight with duality. 
So
\begin{equation*}\label{1}
(P^0(x.\lambda):\zd(w.\lambda))=[\zd(w.\lambda):L(x.\lambda)]\neq 0.
\end{equation*}
Since $xs<x$, we have by Assumption \ref{assum} $\tothe \zp(x.\lambda)=0$, hence $\tothe \zd(w.\lambda)=0$. 
This, together with Proposition \ref{tothe}, implies $ws<w$, that is, $s\in R(w)$.
\end{proof}

\begin{cor}\label{irrfromthe}
We have $$L\fromthe L(w.\mu)\cong L(w.\lambda)\cong R\fromthe L(w.\mu)$$ for $w\in W^+(\mu)$.
\end{cor}
\begin{proof}
We only prove the claim for $L\fromthe$. 
Let $w\in W^+(\mu)$. 
We want to show that $L(w.\lambda)$ is the only quotient of $\fromthe L(w.\mu)$ that is an $A_0$-module. 
By \cite[II.7.20]{J}, any non-simple, non-trivial quotient of $\fromthe L(w.\lambda)$ has a composition factor $L(x.\lambda)$ with $xs<x$.
Thus if $L\fromthe L(w.\lambda)$ is not irreducible, then it contains such a composition factor.
This contradicts Proposition \ref{newlinkage}:
Since the functor $L$ (and thus $L\fromthe$) is right exact, the surjective map $\zd(w.\mu)\surj L(w.\mu)$ induces a surjective map $\zd(w.\lambda)\cong L\fromthe\zd(w.\mu)\surj L\fromthe L(w.\mu)$. 
\end{proof}

Proposition \ref{newlinkage} decomposes $\azmd$ into the ``right descent set linkage classes''.

\begin{prop}\label{regdecomp}
We have a decomposition 
$$\azmd=\bigoplus_{I\subset S_p}\mathcal C_I$$
where $\mathcal C_I$ is the Serre subcategory of $\azmd$ generated by $\{L(w.\lambda)\}_{R(w)=I}$.
\end{prop}
\begin{proof}
Suppose \[\Ext^1_{A^\lambda_0}(L(w.\lambda),L(x.\lambda))\neq 0,\] that is, there is a non-split short exact sequence 
\[0\to L(x.\lambda) \to M\to L(w.\lambda)\to 0,\]
where $M$ is some $A_0^\lambda$-module. 
By the linkage principle, 
we have either $x>w$ or $w>x$. 
If $x<w$, then $M$ is a quotient of $\zd(w.\lambda)$, hence $\zd(w.\lambda)$ has $L(x.\lambda)$ as a composition factor. 
Proposition \ref{newlinkage} shows that $R(x)=R(w)$.
If $w<x$, then $M$ is a submodule of $\zn(x.\lambda)$, and we can apply the dual of Proposition \ref{newlinkage} to get $R(x)=R(w)$.

Therefore, there is no extension between objects in $\mathcal C_I$ and objects in $\mathcal C_{I'}$ with $I\neq I'$, and we have the desired decomposition.
\end{proof}

\begin{lem}\label{rev1}
For $w\in W^+$, $R(w)\subseteq J$, we have $$\zp(w.\mu)\surj\tothe\zp(w.\lambda).$$ 
In particular, $\tothe\zp(w.\lambda)$ is indecomposable. 
\end{lem}
\begin{proof}
By Proposition \ref{tr} (5), we have the surjection 
$$P(w.\mu)^{\oplus N}\cong\tothe P(w.\lambda)\surj\tothe\zp(w.\lambda),$$ 
for some integer $N$.
But since 
\begin{align*}
\Hom_{A_0^\mu} (\tothe\zp(w.\lambda),\zn(w.\mu))&\cong\Hom_{A^\lambda_0}(\zp(w.\lambda),R\fromthe\zn(w.\mu))\\
&\cong \Hom_{A^\lambda_0}(\zp(w.\lambda),\zn(w.\lambda))\\
&\cong k
\end{align*}
by Proposition \ref{fromthe}, 
we have to have $$P(w.\mu)\surj\tothe\zp(w.\lambda).$$ 
Since $\tothe\zp(w.\lambda)$ is an $A_0$-module, this map factors through $\zp(w.\mu)$.
\end{proof}

\begin{prop}\label{into the singular category}
The translation functor $\tothe$ induces an embedding of categories
\[\bigoplus_{I\subseteq J}\mathcal C_I\xrightarrow \tothe \bzmd\]
where $J=S_p(\mu)$. The left inverse is given by the functor $L\fromthe$.
\end{prop}
\begin{proof}
By Lemma \ref{rev1} $A_0$-module $\tothe\zp(w.\lambda)$ has a projective cover $\zp(w.\mu)$. Applying the right exact functor $L\fromthe$ to the surjective map $\zp(w.\mu)\surj\tothe\zp(w.\lambda)$ and using Proposition \ref{fromthep}, we have
\begin{equation}\label{s}\tag{*}
\zp(w.\lambda)\cong L\fromthe\zp(w.\mu)\surj L\fromthe\tothe\zp(w.\lambda).
\end{equation}
 We claim that this is an isomorphism. 
Since $\zp(w.\lambda)$ is an $A_0$-module, it is enough to show that $\zp(w.\lambda)$ is a quotient of $\fromthe\tothe\zp(w.\lambda)$. In fact, the adjoint pair $(\fromthe,\tothe)$ gives the natural transformation $$\eta:\fromthe\tothe\to\id_{\amd},$$ providing the commutative diagram
$$\begin{CD}
\fromthe\tothe P(w.\lambda)\cong P(w.\lambda)^{\oplus N}@>\eta_{P(w.\lambda)}>>P(w.\lambda)\\
@VV\fromthe\tothe(q) V@VVq V\\
\fromthe\tothe\zp(w.\lambda) @>\eta_{\zp(w.\lambda)}>>\zp(w.\lambda)
\end{CD}$$
with surjective upper horizontal and vertical maps. 
This shows that $\eta_{\zp(w.\lambda)}$ is surjective. Since $\zp(w.\lambda)$ and $L\fromthe\tothe\zp(w.\lambda)$ both have irreducible heads, any map of the form \eqref{s} should be an isomorphism. This proves our claim.

Now consider the adjoint pair $(L\fromthe,\tothe)$ and the natural transformation $$L\eta:L\fromthe\tothe\to\id_{\azmd} .$$ What we showed in the previous paragraph actually give
\begin{equation}\label{eqproj}
L\fromthe\tothe\zp(w.\lambda)\xrightarrow[L\eta_{\zp(w.\lambda)}]{\cong}\zp(w.\lambda).
\end{equation}
Using this, we to show that $L\eta$ restricts to an equivalence of functors. 
Let $M\in\bigoplus_{I\subseteq J}\mathcal{C}_I$ and take a projective cover $\zp$ of $M$ in $\azmd$.
Then $P^0\in\bigoplus_{I\subseteq J}\mathcal{C}_I$.
So there is a short exact sequence $$0\to K\to\zp\to M\to 0,$$ where $K$ is the kernel of the map $P^0\to M$. Since $L\fromthe\tothe$ is right exact, we have a commutative diagram
$$\begin{CD}
0 @>>> K @>>> \zp @>>>M @>>> 0\\
@. @AL\eta_{K} AA @AL\eta_{\zp} A\cong A @AL\eta_M AA @.\\
@. L\fromthe\tothe K @>>> L\fromthe\tothe\zp @>>>L\fromthe\tothe M @>>> 0
\end{CD}$$ in $\azmd$. That the middle vertical map is an isomorphism follows from \eqref{eqproj}. The diagram shows that $L\eta_M$ is surjective. 
But since $M$ was arbitrary, $L\eta_K$ is surjective, which implies $L\eta_M$ is injective. 

This establishes $L\fromthe\circ \tothe\cong \id_{\bigoplus_{I\subseteq J}\mathcal{C}_I}$.
\end{proof}

\begin{cor}
The following statements are equivalent.
\begin{enumerate}
\item The category $\bzmd$ is highest weight.
\item The embedding \[\bigoplus_{I\subseteq J}\mathcal C_I\xrightarrow \tothe \bzmd\] from Proposition \ref{into the singular category} is an equivalence of (highest weight) categories.
\item The category $\bzmd$ is equivalent to a direct summand of $\azmd$.
\end{enumerate}
\end{cor}
\begin{proof}
(2) $\Rightarrow$ (3) is obvious, and (3) $\Rightarrow$ (1) is because $\azmd$ is highest weight. Let us assume (1) and prove (2).

For $w\in W^+, R(w)\subseteq J$ the module $\tothe\zp(w.\lambda)$ has a $\zd$-filtration by Proposition \ref{tothe} (and exactness of $\tothe$).
In fact, Proposition \ref{tothe} shows that 
\[(\tothe\zp(w.\lambda):\zd(x.\mu))=(\zp(w.\lambda):\zd(x.\lambda))\] for each $x\in W^J$.
On the other hand, since $\bzmd$ is a highest weight category, the projective module $\zp(w.\mu)$ has a $\zd$-filtration and the multiplicity satisfies
\[(\zp(w.\mu):\zd(x.\mu))=[\zn(x.\mu):L(w.\mu)].\]
But by Proposition \ref{tothe} and Proposition \ref{tr}(3), we have
\begin{align*}
[\zn(x.\mu):L(w.\mu)]&=[\zn(x.\lambda):L(w.\lambda)].
\end{align*}
Finally using that $\azmd$ is highest weight, we conclude
\[(\tothe\zp(w.\lambda):\zd(x.\mu))=(\zp(w.\mu):\zd(x.\mu)).\]
Since $\tothe\zp(w.\lambda)$ is a quotient of $\zp(w.\mu)$ by Lemma \ref{rev1}, this shows that 
\[\tothe \zp(w.\lambda)\cong \zp(w.\mu).\] 

Now we have 
\[\tothe L\fromthe \zp\cong \zp \] for any projective $A_0^\mu$-module $P^0$. 
Then $\tothe L\fromthe\cong \id_{\bzmd}$ follows from the same argument as in the proof of Proposition \ref{into the singular category}.
This gives (2).
\end{proof}

Proposition \ref{newlinkage} can be stated in terms of characters. Since $\ch\zd(w.\lambda)=\ch L_\zeta(w.\lambda)$, we can write the character as a sum or irreducible characters:
\begin{equation}\label{chlinkage}
\ch L_\zeta(w.\lambda)=\sum_{x\leq w}c_{x,w}\ch L(x.\lambda),
\end{equation} with $c_{x,w}\in \Z_{\geq 0}$.
Then Proposition \ref{newlinkage} tells us that $c_{x,w}=0$ unless $R(x)=R(w)$.
By Lusztig's character formula for quantum groups at roots of unity (proved by works of Kazhdan-Lusztig, Lusztig, Kashiwara-Tanisaki, and Andersen-Jantzen-Soergel), we know that 
\[\ch L_\zeta(w.\lambda)=\sum_y (-1)^{l( w)-l( y)} P_{ y,  w}(-1) \chi( y. \lambda),\]
where $P_{y,w}$ is the Kazhdan-Lusztig polynomial associated to the affine Weyl group $W_p$.
So the $c_{x,w}$ in \eqref{chlinkage} determine the irreducible characters for $\operatorname{Rep}(G)$ which is not known in many cases. 
If $p$ is large enough, the irreducible $G$-characters are computed by Lusztig's conjecture for algebraic groups proved in \cite{AJS}, while there are some $p$ (can be much larger than the Coxeter number), found in \cite{wilcountex}, where Lusztig's conjecture does not hold. 
In the latter case, Proposition \ref{newlinkage}, or even Proposition \ref{mult} is a lot of information on irreducible characters. 
There is a more general conjecture by Riche-Williamson \cite{RW} that provides an indirect way to compute the characters using the $p$-Kazhdan-Lusztig polynomials, but these are far less understood than the Kazhdan-Lusztig polynomials, and we do not know the relation to Propositions \ref{newlinkage}, \ref{mult}.
We do not know, for example, whether the Riche-Williamson conjecture implies Proposition \ref{newlinkage}.

The formula in Lusztig's conjecture for algebraic groups is of the same form as in the quantum case, but a restriction on the weights need to be added. 
The reader may compare the two Steinberg tensor product theorems \eqref{steinG} and \eqref{st quan} to see this.
We avoid talking about the weight restriction by presenting the conjecture in different words. 
\begin{lus}
For a regular weight $\lambda'\in X^+$, write $\lambda'=\lambda_0+p\lambda_1$ with $\lambda_0\in X_1=\{\gamma\in X^+\ |\ \langle \gamma
+\rho,\alpha^\vee\rangle <p,\ \ \forall \alpha\in \Pi\}$ and $\lambda_1\in X^+$, the following isomorphism of $G$-modules holds.
\begin{equation*}
\zd(\lambda')\cong L(\lambda_0)\otimes_k \Delta(\lambda_1)^{[1]} 
\end{equation*}
(The module on the left hand side is often called $\Delta^p(\lambda')$. In general, we have $\zd(\lambda')\surj L(\lambda_0)\otimes_k \Delta(\lambda_1)^{[1]}$.)
\end{lus}

This formulation uses the following analogue of Steinberg's tensor product theorem (see \cite[Theorem 2.7]{Lin} or \cite[Proposition 1.7]{CPS7}):
\begin{equation}
\zd(\lambda_0+p\lambda_1)\cong\zd(\lambda_0)\otimes\Delta(\lambda_1)^{[1]}
\end{equation}

For the rest of the section, we assume that $p$ is such that Lusztig's conjecture is true. Recall that this is true if $p$ is large enough.

\begin{prop}\label{lcfimplies}
If $p$ is such that Lusztig's conjecture is true for $G$, then Proposition \ref{newlinkage} is true (without Assumption \ref{assum}).
\end{prop}
\begin{proof}
For $w\in W^+$ write \[w.\lambda=\lambda_0+p\lambda_1\] as above.
By the assumption on $p$,
\[\Delta^0(w.\lambda)\cong L(\lambda_0)\otimes_k \Delta(\lambda_1)^{[1]}.\]
Note that $\Delta(\lambda_1)^{[1]}$ has composition series with factors $L(\gamma)^{[1]}$ where $L(\gamma)$ are composition factors of $\Delta(\lambda_1)$. In particular, $\lambda_1 -\gamma$ is in the root lattice $\Z R$. (For any $\gamma\in X$ and a simple root $\alpha$, the difference between $s_{\alpha,m}.\gamma$ and $\gamma$ is in $\Z\alpha$ for any $m\in \Z$. So the statement follows from the linkage principle.)
Since \[L(\lambda_0)\otimes_k L(\gamma)^{[1]}\cong L(\lambda_0+p\gamma)\] is simple, the composition factors of $L(\lambda_0)\otimes_k \Delta(\lambda_1)^{[1]}$ are all of this form.

This shows that if $x\in W^+$ is such that $[\zd(w.\lambda):L(x.\lambda)]\neq 0$, then the weight $x.\lambda$ is of the form
\[x.\lambda=\lambda_0+p\gamma.\]
Thus, \[w.\lambda=t(x.\lambda)=t((x.\lambda)+\rho)-\rho=tx.\lambda\]
where $t=t_{p(\lambda_1-\gamma)}\in W$ is the translation by ${p(\lambda_1-\gamma)}\in  p\Z R$.
The proof is complete using the elementary Lemma \ref{elem} below.
\end{proof}

\begin{lem}\label{elem}Let $w\in W^+$ and $t\in W_p$ be a translation by some element $p\nu$ in $p\Z R$. If $tw\in W^+$, then $R(w)=R(tw)$.  \end{lem}
\begin{proof}
Since $w\in W^+$, the condition $ws>w$ for $s\in S_p$ is equivalent to 
\[w.\lambda \uparrow ws.\lambda ,\]
where the relation $\uparrow $ is the ordering on $X=X(T)$ generated by $\uparrow_{s'}$ for all reflections $s'$ in $W_p$, (in other words, $s'=s_{\alpha,m}$ for some $\alpha\in R^+, m\in p\Z$) where 
\[\gamma \uparrow_{s'}\gamma' \Leftrightarrow \gamma<s'.\gamma=\gamma'.\]
(This is well-known, but we explain this in \S\ref{appendix}. The above statement is Proposition \ref{ordersame} applied to $C=C^-$.)
In fact, we have 
\[w.\lambda\uparrow_{s'}ws.\lambda \]
for some $s'=s_{\alpha,m}$.
Then one can check
\[w.\lambda+p\nu\uparrow_{s''} ws.\lambda+p\nu\] for $s''=s_{\alpha,(m+\langle p\nu,\alpha^\vee\rangle)}$.

This shows that
\[tw.\lambda=t(w.\lambda)=w.\lambda+p\nu\uparrow ws.\lambda+p\nu=t(ws.\lambda)=tws.\lambda,\]
which implies $tw<tws,$ again by Proposition \ref{ordersame}.

Therefore, we have $R(w)\supset R(tw)$. The same argument shows the other inclusion.
\end{proof}

\begin{cor}
If $p$ is such that Lusztig's conjecture is true for $G$, then Assumption \ref{assum} is satisfied.
\end{cor}
\begin{proof}
Let $w$ be as in Assumption \ref{assum}, that is, $w\not\in W^J$.
It is enough to show that each section of a $\Delta^0$-filtration of $P^0(w.\lambda)$ is mapped to zero under the exact functor $\tothe$. Let $\Delta^0(x.\lambda)$ be such a (nonzero) section.
Since
\[(P^0(w.\lambda):\Delta^0(x.\lambda))=[\Delta^0(x.\lambda):L(w.\lambda)],\]
Proposition \ref{lcfimplies} implies $R(x)=R(w)$. 
In particular, $x\not\in W^J$. 
Now Proposition \ref{tothe} shows that $\tothe\Delta^0(x.\lambda)=0$.
\end{proof}

\begin{appendix}
\section{Affine Weyl group and ordering on the weights}\label{appendix}
We explain here why we take the top antidominant alcove $C^-$ (the one that contains $-2\rho$) for the fundamental domain of the affine Weyl group action, as opposed to the bottom dominant alcove $C^+$ (which contains 0). We also make explicit some combinatorics of the affine Weyl group used in the paper, providing precise definitions of the objects involved.

For each $\alpha\in R^+$ and $m\in \Z$, the affine reflection $s_{\alpha,mp}\in W_p$ acts on $X\otimes_\Z\R$ as
\[s_{\alpha,mp}.(v)=v-(\langle v+\rho,\alpha^\vee\rangle-mp)\alpha\] for $v\in X\otimes_\Z\R$. 
The reflection hyperplanes divides $X\otimes_\Z\R$ into alcoves, which are by definition subsets of $X\otimes_\Z \R$ of the form
\begin{equation}\label{alcove}
C=\{v\in X\otimes_\Z\R\ |\ (m_\beta -1)p<\langle v+\rho,\beta^\vee\rangle <m_\beta p,\ \forall \beta\in R^+\},
\end{equation}
 where $m_\beta=m_\beta^C\in\Z$.
For the alcoves $C^+$ and $C^-$, we have $m^{C^+}_\beta=1$ for all $\beta\in R^+$, and $m^{C^-}_\beta=0$ for all $\beta\in R^+$.
For each alcove $C$, the reflections through walls of $C$ form a generator set for the Coxeter group $W_p$. We denote this generating set by $S_p(C)$. Each $\overline C\cap X$ is a fundamental domain for the dot action of $W_p$ on $X$.
Thus, the Coxeter ordering ``$\leq$'' on ($W_p,S_p(C)$) induces an order relation on $X$.
Let us first consider generally all such order relations:
\begin{defn}
For each alcove $C$, the partial order relation $\leq_C$ on $X$ is defined as follows.
For $\lambda,\mu\in X$, we say $\lambda\leq_C\mu$ if there exists a weight $\nu\in \overline C$ such that $\lambda=w.\nu$ and $\mu=x.\nu$ with $w\leq x\in W_p$.
\end{defn}

Now fix an alcove $C$ and define two functions $d=d_C$ and $d'=d'_C$ on the alcoves by
\[d(C')=\sum_{\alpha\in R} (m^{C'}_\alpha-m^C_\alpha),\]
and
\[d'(C')=\sum_{\alpha\in R} |m^{C'}_\alpha-m^C_\alpha|.\]
So $d'(C')$ is the number of reflection hyperplanes that separate $C$ and $C'$. 
If $C'\subset X^++C$, then $m^{C'}_\alpha-m^C_\alpha\geq 0$ for all $\alpha\in R^+$, and hence $d(C')=d'(C')$.
An easy fact is the following lemma whose proof is left to the reader.
\begin{lem}\label{exercise}
We have $l(w)=d'(w.C)$ for all $w\in W_p$.
\end{lem}

Similarly define $d(\lambda), d'(\lambda)$ for $\lambda\in X$ using the fact that $\lambda$ belongs to the upper closure of a unique alcove $C'$, that is, let $d(\lambda)=d(C')$ and $d'(\lambda)=d'(C')$ for such $C'$.

\begin{defn}\label{uparrow}
 For each reflection $s'=s_{\alpha,m}$, where $\alpha\in R^+, m\in p\Z$, define $\uparrow_{s'}$ by 
\[\gamma \uparrow_{s'}\gamma' \Leftrightarrow \gamma<s'.\gamma=\gamma'.\]
The uparrow ordering $\uparrow$ on $X$ is defined as the order relation generated by all such $\uparrow_{s'}$. (Here ``$<$'' is the dominance order. That is, $\gamma<\gamma'$ if and only if $\gamma'-\gamma$ is a linear combination of fundamental weights with positive coefficients.)
\end{defn}

Now we compare the two order relations.
\begin{prop}\label{ordersame}
Let $C$ be any alcove. Suppose $\lambda,\mu\in X^++\overline C$. Then we have $\lambda\leq_C\mu$ if and only if $\lambda\uparrow\mu$.
\end{prop}
\begin{proof}
It is enough to consider the generating case: Let $\lambda=w.\nu$ and $\mu=ws.\nu$ for $\nu\in\overline C$ and $s\in S_p(C)$. We can also assume that $\lambda\neq\mu$, so $s$ is not a stabilizer of $\nu$. Then one finds unique $\alpha\in R^+$ and $m\in \Z$ such that $s_{\alpha,mp}.(w.\nu)=s_{\alpha,mp}w.\nu=ws.\nu$. 
For all $\alpha\neq\beta\in R^+$, there are $m_\beta\in\Z$ such that
\[(m_\beta -1)p\leq\langle w.\nu+\rho,\beta^\vee\rangle,\langle ws.\nu+\rho,\beta^\vee\rangle\leq m_\beta p,\]
and for $\alpha$, we have either
\begin{equation}\label{first}
(m_\alpha-1)p<\langle w.\nu+\rho,\alpha^\vee\rangle<m_\alpha p<\langle ws.\nu+\rho,\alpha^\vee\rangle<(m_\alpha+1)p
\end{equation}
or
\begin{equation}\label{second}
(m_\alpha-1)p<\langle ws.\nu+\rho,\alpha^\vee\rangle<m_\alpha p<\langle w.\nu+\rho,\alpha^\vee\rangle<(m_\alpha+1)p.
\end{equation}
Here \eqref{first} is when $w.\nu<ws.\nu$, or $\lambda\uparrow\mu$, and \eqref{second} is when $ws.\nu<w.\nu$, or $\mu\uparrow\lambda$.
Since both $w.\nu$ and $ws.\nu$ are in $X^++\overline C$, we know that $m_\beta\geq m_\beta^C$ where $m^C_\beta$ is as in \eqref{alcove}.
This implies $d(w.C)=d'(w.C)$ and $d(ws.C)=d'(ws.C)$.

Now, $w.\nu=\lambda<_C\mu=ws.\nu$ is equivalent to $l(w)<l(ws)$ and, by Lemma\ref{exercise}, to $d'(w.C)<d'(ws.C)$. 
This is the case where \eqref{first} is true, which is equivalent to $\lambda\uparrow\mu$.
\end{proof}
\begin{rem}
The condition in the proposition is about the minimal condition for the two order relations to coincide. For example, if $\lambda=e.\nu\in \overline C$ and $\mu=s.\nu\not\in X^++\overline C$ for some $s\in S_p(C)$ (there always exists such an $s$) then $\mu\leq\lambda$ and $e\leq s$. Similar things happen along the boundary of the shifted dominant region $X^++\overline C$.
\end{rem}

Note that $X^+=X^++C^+\subset X^++\overline{C^+}$. So the uparrow relation and $\leq_{C^+}$ agree in the dominant cone. But as remarked above, they do not agree near the boundary of the dominant cone. This is not desirable when we consider right descent sets. To determine the right descent set of an element $w\in W_p$, we compare $w$ with $ws$ for each $s\in S(C)$. It is possible that $w\in W^+$ while $ws\not\in W^+$, so $ws.C$ is out of the dominant region. In that case, we have
$T_s L(w.0)=0$ while $s\not\in R(w)$ and $T_s=T_0^\mu$ is a translation to the $s$-wall. So we do not want to say, for example,
\[T_0^\mu L(w.0)=\begin{cases}
L(w.\mu)\ \ \ \ \ \ \ \ \ \text{if } s\not\in R(w)\\
0\ \ \ \ \ \ \ \ \ \ \ \ \ \ \ \ \ \text{if } s\in R(w)
\end{cases},\]
because $w.\mu$ may not be dominant while $s\not\in R(w)$. 
In particular, if we indexed the weights using $C^+$ in this paper, Proposition \ref{mult} is absurd. 

In the paper, we used $C^-$ to index the weights, hence the order relation $\leq=\leq_{C^-}$. 
As a consequence, for $w\in W^+$, we have $ws.C\in X^++C^-$ for each $s\in S=S(C^-)$. So we can apply Proposition \ref{ordersame} to have
\[s\in R(w)\Leftrightarrow ws.C\uparrow w.C ,\]
and 
\[T_\lambda^\mu L(w.\lambda)=\begin{cases}
L(w.\mu)\neq 0 \ \ \ \ \text{if } s\not\in R(w)\\
0\ \ \ \ \ \ \ \ \ \ \ \ \ \ \ \ \ \text{if } s\in R(w)
\end{cases}\]
is true where $\lambda\in C^-$.
This works with any other antidominant alcove instead of $C^-$, but those will make the correspondence between $W_p$ (or ``$W^+$'') and weights rather complicated. 

An easy example to compare the two choices for the fundamental domain, $C^+$ and $C^-$, is the following. Consider the trivial module $L(0)$. 
We write it as $L(e.0)$ with $0\in C^+$, or as $L(w_0.-2\rho)$ with $-2\rho\in C^-$. (Recall that $w_0$ is the longest element in the finite Weyl group $W_f\leq W_p$.)
The right descent sets in the two cases are $R(e)=\emptyset$ and $R(w_0)=S_p\setminus S_f$, respectively.
This shows that the statement for Proposition \ref{mult} cannot be as concise as it is now, if we index the weights using $C^+$. As written, it would claim that $L(0)$ is not a composition factor of any $\zd(\gamma)$. But, for example, $\zd(p(s_0.0))$ contains $L(0)$ since it surjects to $\Delta(s_0.0)^{[1]}\cong L(0)\otimes_k\Delta(s_0.0)^{[1]}$ where $s_0=s_{\alpha_0,p}\in S_p(C^+)$.
\end{appendix}

\bibliographystyle{abbrv}
\end{document}